\documentclass[11pt]{amsart}

\usepackage[utf8]{inputenc}
\usepackage{amsmath,amsxtra,amssymb,latexsym,amscd,amsthm}
\usepackage[linesnumbered,ruled,vlined]{algorithm2e}
\usepackage{csquotes}

\usepackage{tikz-cd}
\usepackage{tkz-euclide}

\newtheorem{theorem}{Theorem}[section]

\theoremstyle{definition}

\newtheorem{example}[theorem]{Example}

\theoremstyle{remark}

\numberwithin{equation}{section}

\theoremstyle{break}
\newtheorem{defi}{Definition}[section]
\newtheorem{thm}{Theorem}[section]
\newtheorem{prop}{Proposition}[section]
\newtheorem{lem}{Lemma}[section]

\theoremstyle{remark}
\newtheorem{rem}{Remark}[section]
\newtheorem*{ackn}{Acknowledgements}

\newcommand{\Sp}{\operatorname{Spec}}
\newcommand{\pr}{\text{Proj}}
\newcommand{\spf}{\text{Specf}}
\newcommand{\F}{\mathbb{F}_2}

\newcommand{\proj}{\operatorname{pr}}
\newcommand{\Aut}{\operatorname{Aut}}
\begin{document}

\title{A counter-example to the equivariance structure on semi-universal deformation}

\author{ An Khuong DOAN}
\address{An Khuong DOAN, IMJ-PRG, UMR 7586, Sorbonne Université,  Case 247, 4 place Jussieu, 75252 Paris Cedex 05, France}
\email{an-khuong.doan@imj-prg.fr }
\thanks{ }


\date{January 22, 2019.}

\dedicatory{ }


\begin{abstract} If $X$ is a projective variety and $G$ is an algebraic group acting algebraically on $X$, we provide a counter-example to the existence of a $G$-equivariant extension on the formal semi-universal deformation of $X$.
\end{abstract}

\maketitle
\tableofcontents

\section*{Introduction}
Let $X$ be an algebraic variety defined over a field $k$ of characteristic zero. Due to Schlessinger's work in [5], the existence of a formal semi-universal deformation (unique up to non-canonical isomorphism), which contains all the information of small deformations of $X$, is assured provided that $H^1(X,\mathcal{T}_X)$ and $H^2(X,\mathcal{T}_X)$ are finite dimensional vector spaces. These conditions realise for example, if $X$ is a complete scheme over $k$ or an affine scheme with at most isolated singularities (see [6, Corollary 2.4.2]). Now, we equipe $X$ with an action of an algebraic group $G$ defined over $k$. One question arising naturally is whether there exists a formal semi-universal deformation $\pi:$ $\mathcal{X} \rightarrow S$ of $X$, on which we can provide a $G$-action extending the given one on $X$. The answer is positive in the case that $G$ satisfies some vanishing conditions on its cohomology groups, i.e. $H^1(G,-)=0$ and $H^2(G,-)=0$ for a class of $G$-modules determined by $X$. In particular, these vanishing conditions hold for linearly reductive groups (see [4]). However, we do not know if there exists a non-reductive group whose action on $X$ does not extend to the formal semi-universal deformation of $X$. Therefore, we wish to give an example which illustrates this phenomenon. More precisely, we prove that the action of the automorphism group of the second Hirzebruch surface $\F$ does not extend to its formal semi-universal deformation. 

Our proof goes as follows. First, we find a nice presentation of $G:=\Aut(\F)$. Then we construct a formal semi-universal deformation $\widehat{\mathcal{X}}$ of $\F$. It turns out that $G$ is non-reductive and that the Lie algebra of $G$ is a $7$-dimensional vector space. As a matter of fact, we obtain seven vector fields on $\F$ with Lie bracket relations induced by those in $\text{Lie}(G)$. Next, we describe the general form of formal vector fields on $\widehat{\mathcal{X}}$. Finally, we conclude the paper by means of contradiction. Suppose that the $G$-action on $\F$ does extend to a $G$-action on $\widehat{\mathcal{X}}$ then we also have seven formal vector fields on $\widehat{\mathcal{X}}$ whose restrictions on the central fiber are nothing but our initial ones on $\F$. By manipulating these vector fields with a filtration $F$ given by the vanishing order at $0$, we obtain the existence of a $3$-dimensional abelian Lie subalgebra in $\mathfrak{sl}_2(K)\times \mathfrak{sl}_2(K)$, where $\mathfrak{sl}_2(K)$ is the special linear group and $K$ is the field of formal Laurent power series $\mathbb{C}[[ t,t^{-1} ]]$, which is not the case. A remark is in order. Since the semi-universal family of $\F$ is in fact not universal, another possible way to obtain a contradiction is to use Wavrik's criterion (see [7, Theorem 4.1]) but the calculations are rather complicated.
\begin{ackn} I would like to thank Prof. Bernd Siebert for many useful discussions. Actually, I learned the idea of using the extension of vector fields and their relations as obstructions to the extension of the group action from an unpublished paper of his. This provides a strategy to attack the problem. I am specially thankful to Prof. Julien Grivaux for his careful reading and his comments which help to improve the manuscript. Finally, I am warmly grateful to the referee whose work led to a remarkable improvement of the paper.

\end{ackn}
\section{Formal schemes and formal deformations}
In this section, by $k$, we always mean a field of characteristic zero. We begin by recalling the definition of formal schemes. For more details, the readers are referred to [2, Chapter III. 9].
\begin{defi}
Let $X$ be a noetherian scheme and let $Y$ be a closed subscheme defined by a sheaf of ideals $\mathcal{J}.$ Then we define the formal completion of $X$ along $Y$, denoted $(\widehat{X},\mathcal{O}_{\widehat{X}})$ (sometimes just $\widehat{X}$), to be the following ringed space. We take the topological space $Y$, and on it the sheaf of rings $\mathcal{O}_{\widehat{X}}=\varprojlim\mathcal{O}_{X}/\mathcal{I}^n$. Here we consider each $\mathcal{O}_{X}/\mathcal{I}^n$ as sheaf of rings on $Y$
\end{defi} 
\begin{rem} For each $n$, let $X_n=(X,\mathcal{O}_{X}/\mathcal{I}^n)$. Then we obtain a sequence of closed immersions of schemes $$X_1\rightarrow X_2\rightarrow \cdots\rightarrow X_n\rightarrow\cdots.$$
This expression is helpful in the sequel.
\end{rem}
\begin{defi}
A noetherian formal scheme is a locally ringed space $(\mathfrak{X},\mathcal{O}_\mathfrak{X})$ which has a finite open cover $\lbrace \mathfrak{U}_i\rbrace$ such that for each $i$, the pair $(\mathfrak{U}_i,\mathcal{O}_\mathfrak{X}\mid_{\mathfrak{U}_i})$ is isomorphic, as a locally ringed space, to the completion of some noetherian scheme $X_i$ along a closed subscheme $Y_i$. A morphism of noetherian formal schemes is a morphism as locally ringed spaces.
\end{defi}
\begin{example} If $X$ is any noetherian scheme, and $Y$ is a closed subscheme then the formal completion $\widehat{X}$ of $X$ along $Y$ is a formal scheme.
\end{example}
\begin{example} For $X=\mathbb{C}^1=\Sp(\mathbb{C}[t])$ and $Y=\lbrace 0 \rbrace$, the formal scheme $\widehat{X}$ is the locally ringed space $(Y,\mathcal{O}_{\widehat{X}})$, where the structure sheaf $\mathcal{O}_{\widehat{X}}$ is $\mathbb{C}[[t]]$. We denote $\text{Specf}(\mathbb{C}[[t]]):=(Y,\mathcal{O}_{\widehat{X}})$.
\end{example}Let $(\mathfrak{X},\mathcal{O}_\mathfrak{X})$ be a noetherian formal scheme. We would like to define \textit{formal vector fields} on $\mathfrak{X}$. Let $\lbrace \mathfrak{U}_i\rbrace$ be a finite open cover  of $\mathfrak{X}$ such that for each $i$, the pair $(\mathfrak{U}_i,\mathcal{O}_\mathfrak{X}\mid_{\mathfrak{U}_i})$ is the formal completion $(\widehat{X_i},\mathcal{O}_{\widehat{X}_i})$ of some noetherian scheme $X_i$ along a closed subscheme $Y_i$. By Remark 1.1, for each $i$ we have a sequence of closed immersions of schemes $$X_{i1}\rightarrow X_{i2}\rightarrow \cdots\rightarrow X_{in}\rightarrow\cdots.$$ 
\begin{defi} A formal vector field on a noetherian formal scheme $\mathfrak{X}$ is a sequence of vector fields $\lbrace v_{i,n} \rbrace$ such that 
\begin{enumerate} 
\item[($i$)] Each $v_{i,n}$ is a usual vector field on the scheme $X_{i,n}$, 
\item[($ii$)] $v_{i,n}$ induces $v_{i,n-1}$ via the natural inclusion $X_{i,n-1} \rightarrow X_{i,n}$,
\item[($iii$)] $v_{i,n}\mid_{X_{i,n} \cap X_{j,n}}= v_{j,n}\mid_{X_{j,n} \cap X_{i,n}}$.
\end{enumerate}
\end{defi}
Next, we turn to the notion of infinitesimal deformations and that of formal deformations. Let $X$ be an algebraic scheme and let $A$ be an artinian local $k$-algebra with residue field $k$. An infinitesimal deformation of $X$ is a deformation of $X$ over the scheme $\Sp(A)$, i.e. a commutative diagram
\begin{center}
\begin{tikzpicture}[every node/.style={midway}]
  \matrix[column sep={8em,between origins}, row sep={3em}] at (0,0) {
    \node(Y){$X$} ; & \node(X) {$\mathcal{X}$}; \\
    \node(M) {$\Sp(k)$}; & \node (N) {$\Sp(A)$};\\
  };
  
  \draw[->] (Y) -- (M) node[anchor=east]  {}  ;
  \draw[->] (Y) -- (X) node[anchor=south]  {$i$};
  \draw[->] (X) -- (N) node[anchor=west] {$\pi$};
  \draw[->] (M) -- (N) node[anchor=north] {};.
\end{tikzpicture}
\end{center}
where $\pi:\;\mathcal{X}\rightarrow \Sp(A)$ is a flat surjective morphism of schemes. 

Now, let $A$ be a complete local noetherian $k$-algebra with the unique maximal ideal $\mathfrak{m}$ and with residue $k$.
\begin{defi}  A formal deformation of $X$ over $A$ is a sequence $\lbrace\nu_n \rbrace$ of infinitesimal deformations of $X$, in which $\nu_n$ is represented by a deformation
\begin{center}
\begin{tikzpicture}[every node/.style={midway}]
  \matrix[column sep={9em,between origins}, row sep={3em}] at (0,0) {
    \node(Y){$X$} ; & \node(X) {$\mathcal{X}_n$}; \\
    \node(M) {$\Sp(k)$}; & \node (N) {$\Sp(A_n)$};\\
  };
  
  \draw[->] (Y) -- (M) node[anchor=east]  {}  ;
  \draw[->] (Y) -- (X) node[anchor=south]  {$f_n$};
  \draw[->] (X) -- (N) node[anchor=west] {$\pi_n$};
  \draw[->] (M) -- (N) node[anchor=north] {};.
\end{tikzpicture}
\end{center}
 where $A_n =A/\mathfrak{m}^{n+1}$, such that for all $n\geq 1$, $\nu_n$ induces $\nu_{n-1}$ by pullback under the natural inclusion $\Sp(A_{n-1}) \rightarrow \Sp(A_{n})$, i.e. $\nu_{n-1}$ is also represented by the deformation
 \begin{center}
\begin{tikzpicture}[every node/.style={midway}]
  \matrix[column sep={9em,between origins}, row sep={3em}] at (0,0) {
    \node(Y){$X$} ; & \node(X) {$\mathcal{X}_n\times_{\Sp(A_n)}\Sp(A_{n-1})$}; \\
    \node(M) {$\Sp(k)$}; & \node (N) {$\Sp(A_{n-1})$};\\
  };
  
  \draw[->] (Y) -- (M) node[anchor=east]  {}  ;
  \draw[->] (Y) -- (X) node[anchor=south]  {$f_{n-1}$};
  \draw[->] (X) -- (N) node[anchor=west] {$\pi_{n-1}$};
  \draw[->] (M) -- (N) node[anchor=north] {};.
\end{tikzpicture}
\end{center}
\end{defi}
In the language of formal schemes, we can write $\lbrace \nu_n \rbrace$ as the morphism of formal schemes 
$$\widehat{\pi}:\widehat{\mathcal{X}}\rightarrow \spf(A)$$ where 
$$\widehat{\mathcal{X}}=(X,\lim_{\leftarrow}\mathcal{O}_{\mathcal{X}_n})\text{ and }\widehat{\pi} =\lim_{\leftarrow} \pi_n.$$ Here, $\mathcal{O}_{\mathcal{X}_n}$ is the structure sheaf on $\mathcal{X}_n$ and $\spf(A)$ is the formal scheme obtained by completing $\Sp(A)$ along its closed point, which corresponds to the unique maximal ideal of $A$. The easiest way to construct formal deformations is to build out of usual ones. This leads to the definition of \textit{formal deformation associated to a given deformation}. Let $X$ be a projective scheme and let $\nu$ be a deformation represented by 
\begin{center}
\begin{tikzpicture}[every node/.style={midway}]
  \matrix[column sep={9em,between origins}, row sep={3em}] at (0,0) {
    \node(Y){$X$} ; & \node(X) {$\mathcal{X}$}; \\
    \node(M) {$\Sp(k)$}; & \node (N) {$(S,s)$};\\
  };
  
  \draw[->] (Y) -- (M) node[anchor=east]  {}  ;
  \draw[->] (Y) -- (X) node[anchor=south]  {$f$};
  \draw[->] (X) -- (N) node[anchor=west] {$\pi$};
  \draw[->] (M) -- (N) node[anchor=north] {};.
\end{tikzpicture}
\end{center} where $S=\Sp(B)$ for some $k$-algebra of finite type $B$  and $s$ is a $k$-rational point of $S$.
\begin{defi} The formal deformation associated to $\nu$ is defined to be the sequence of deformations $\lbrace \nu_n \rbrace$ where each $\nu_n$ is the pullback of $\nu$ under the natural closed embedding
$$S_n:=\Sp(\mathcal{O}_{S,s}/\mathfrak{m}_s^{n+1}) \rightarrow S$$ where $\mathfrak{m}_s$ is the unique maximal ideal of the local ring $\mathcal{O}_{S,s}$.

\end{defi}
\begin{rem}
Note that $\lbrace \nu_n \rbrace$ is formal because of the isomorphism 
$$\mathcal{O}_{S,s}/\mathfrak{m}_s^{n+1}\cong \widehat{\mathcal{O}}_{S,s}/\widehat{\mathfrak{m}}_s^{n+1}$$ for all $n$.
\end{rem}

To end this section, we introduce a very interesting kind of (formal) deformations, namely, the kind of $G$-equivariant (formal) ones, which is of central interest of the article. Let $G$ be a $k$-algebraic group acting algebraically on a projective variety $X$ and $A$ an artinian local $k$-algebra.
\begin{defi} A $G$-equivariant infinitesimal deformation of $X$ over $\Sp(A)$ is a usual deformation of $X$, i.e. a commutative diagram \begin{center}
\begin{tikzpicture}[every node/.style={midway}]
  \matrix[column sep={8em,between origins}, row sep={3em}] at (0,0) {
    \node(Y){$X$} ; & \node(X) {$\mathcal{X}$}; \\
    \node(M) {$\Sp(k)$}; & \node (N) {$\Sp(A)$};\\
  };
  
  \draw[->] (Y) -- (M) node[anchor=east]  {}  ;
  \draw[->] (Y) -- (X) node[anchor=south]  {$i$};
  \draw[->] (X) -- (N) node[anchor=west] {$\pi$};
  \draw[->] (M) -- (N) node[anchor=north] {};.
\end{tikzpicture}
\end{center}
where $\mathcal{X}$ and $\Sp(A)$ are equipped with $G$-actions in a way that any map appearing in the above diagram is $G$-equivariant. In particular, the restriction of the $G$-action on $\mathcal{X}$ on the central fiber is nothing but the initial $G$-action on $X$.
\end{defi}

Finally, we give the definition of $G$-equivariant formal deformations and then we show how to produce formal vector fields from $G$-equivariant formal deformations.
 
\begin{defi}   A $G$-equivariant formal deformation of $X$ over a complete local noetherian $k$-algebra $A$ with the unique maximal ideal $\mathfrak{m}$  is a formal deformation of $X$, i.e. a sequence $\lbrace\nu_n \rbrace$ of infinitestimal deformations of $X$, in which $\nu_n$ is represented by a $G$-infinitesimal deformation
\begin{center}
\begin{tikzpicture}[every node/.style={midway}]
  \matrix[column sep={9em,between origins}, row sep={3em}] at (0,0) {
    \node(Y){$X$} ; & \node(X) {$\mathcal{X}_n$}; \\
    \node(M) {$\Sp(k)$}; & \node (N) {$\Sp(A_n)$};\\
  };
  
  \draw[->] (Y) -- (M) node[anchor=east]  {}  ;
  \draw[->] (Y) -- (X) node[anchor=south]  {$f_n$};
  \draw[->] (X) -- (N) node[anchor=west] {$\pi_n$};
  \draw[->] (M) -- (N) node[anchor=north] {};.
\end{tikzpicture}
\end{center}
 where $A_n =A/\mathfrak{m}^{n+1}$, such that for all $n\geq 1$, the $G$-equivariant deformation $\nu_{n}$  induces the $G$-equivariant deformation $ \nu_{n-1}$ by pullback under the natural inclusion $\Sp(A_{n-1}) \rightarrow \Sp(A_{n})$.
\end{defi}

As before, we can write $\lbrace \nu_n \rbrace$ as the $G$-equivariant morphism of formal schemes 
$$\widehat{\pi}:\widehat{\mathcal{X}}\rightarrow \spf(A)$$ where 
$$\widehat{\mathcal{X}}=(X,\lim_{\leftarrow}\mathcal{O}_{\mathcal{X}_n})\text{ and }\widehat{\pi} =\lim_{\leftarrow} \pi_n.$$ Here, the $G$-equivariance of $\widehat{\pi}$ means that $\widehat{\pi}$ is an inverse limit of $G$-equivariant morphisms of schemes $\pi_n$. On one hand, on each $n^{\text{th}}$-infinitesimal neighborhood, $G$-actions on $\mathcal{X}_n$ and on $\Sp(A_n)$ induce vector fields on $\mathcal{X}_n$ and on $\Sp(A_n)$, respectively. They are related by the fact that the differential of $\pi_n$ always maps the former ones to the latter ones. On the other hand, these induced vector fields on $\mathcal{X}_{n}$ and on $\Sp(A_{n})$ are also induced by those on $\mathcal{X}_{n+1}$ and on  $\Sp(A_{n+1})$, respectively, via the natural inclusion $\Sp(A_{n}) \rightarrow \Sp(A_{n+1})$.  Therefore, we obtain \textit{formal vector fields}, induced by the $G$-actions, on $\widehat{\mathcal{X}}$ and on $\spf(A)$, respectively.
\section{The second Hirzebruch surface and its automorphism group} For the rest of the paper, we assume that $k$ is the field of complex numbers $\mathbb{C}$. The geneneral linear group $\text{GL}(2,\mathbb{C})$ has an obvious linear action on $\mathbb{C}^2$. This induces an action on the $\mathbb{C}$-vector space of polynomials in two variables $\mathbb{C}[X,Y]$. Since the subspace of homogeneous polynomials of degree $2$, denoted by $\mathbb{C}[X,Y]_2$, is $\text{GL}(2,\mathbb{C})$-invariant then we have a $\text{GL}(2,\mathbb{C})$-action on $\mathbb{C}[X,Y]_2$. More precisely, for $g=\begin{pmatrix}
a &b \\c 
 &d 
\end{pmatrix} \in \text{GL}(2,\mathbb{C})$ and $f=a_0X^2+a_1XY+a_2Y^2 \in \mathbb{C}[X,Y]_2$, the action of $g$ on $f$ is given by the linear substitution 
$$\begin{pmatrix}
X\\Y 

\end{pmatrix}:=\begin{pmatrix}
a &b \\c 
 &d 
\end{pmatrix}\begin{pmatrix}
X\\Y 

\end{pmatrix},$$ i.e. 
\begin{align*}
 g.f&= a_0\left (aX+bY  \right )^2+a_1\left (aX+bY  \right )\left (cX+dY  \right )+a_2\left (cX+dY  \right )^2 \\ 
 &=\left ( a^2a_0+aca_1+ c^2a_2 \right )X^2+\left ( 2aba_0+(ad+bc)a_1+2cda_2 \right )XY+\left ( b^2a_0+bda_1+ d^2a_2 \right )Y^2 .
\end{align*}
Identifying $\mathbb{C}[X,Y]_2$ with $\mathbb{C}^3$, the corresponding action on $\mathbb{C}^3$ can be written as 
$$g.(a_0,a_1,a_2)= \begin{pmatrix}
a^2 &ac  &c^2 \\2ab 
 &ad+bc  &2cd \\b^2 
 &bd  &d^2 
\end{pmatrix}\begin{pmatrix}
a_0\\a_1 
\\a_2 
\end{pmatrix}.$$ This action gives rise to an algebraic group $H$ which is the semi-product of $\mathbb{C}^3$ and  $\text{GL}(2,\mathbb{C})$, i.e. 
$$H:=\mathbb{C}^{3} \rtimes \text{GL}(2,\mathbb{C}).$$ This is a non-reductive linear group. Recall that an algebraic group $K$ is reductive if the greatest connected normal subgroup $R_u(K)$ of $K$ is trivial. In our case, $R_u(H)= \mathbb{C}^3$.

Next, we recall the definition of the second Hirzebruch surface. Let $\mathbb{P}(\mathcal{O}_{\mathbb{P}^1}(2)\oplus \mathcal{O}_{\mathbb{P}^1})$ be the projectivization of $\mathcal{O}_{\mathbb{P}^1}(2)\oplus \mathcal{O}_{\mathbb{P}^1}$, where $\mathcal{O}_{\mathbb{P}^1}$ is the structure sheaf of the projective space $\mathbb{P}^1$.
\begin{defi} The second Hirzebruch surface is defined to be  $\mathbb{P}(\mathcal{O}_{\mathbb{P}^1}(2)\oplus \mathcal{O}_{\mathbb{P}^1})$.
\end{defi}
\begin{prop} The second Hirzebruch surface is isomorphic to the variety $$\F:=\{ ([x:y:z],[u:v])  \in \mathbb{P}^2 \times \mathbb{P}^1 | yv^2=zu^2\}.$$
\end{prop}
\begin{proof}
Let $\sigma$: $\mathbb{P}(\mathcal{O}_{\mathbb{P}^1}(2)\oplus \mathcal{O}_{\mathbb{P}^1})\rightarrow \mathbb{P}^1$ be the canonical projection of the projectivization $\mathbb{P}(\mathcal{O}_{\mathbb{P}^1}(2)\oplus \mathcal{O}_{\mathbb{P}^1})$, let $U=\Sp(\mathbb{C}[v])$ and $U'=\Sp(\mathbb{C}[v'])$ such that $v'v=1$ on $U\cap U'$. Then $\mathbb{P}(\mathcal{O}_{\mathbb{P}^1}(2)\oplus \mathcal{O}_{\mathbb{P}^1})$ has the following presentation
$$\mathbb{P}(\mathcal{O}_{\mathbb{P}^1}(2)\oplus \mathcal{O}_{\mathbb{P}^1})=\sigma^{-1}(U)\cup \sigma^{-1}(U')=(U\times \mathbb{P}^1) \cup (U'\times \mathbb{P}^1).$$ such that on the intersection of the  affine open sets
$V=\Sp(\mathbb{C}[v,y])\subset U\times \mathbb{P}^1$ and $V'=\Sp(\mathbb{C}[v',y'])\subset U\times \mathbb{P}^1$, we have 
$$\begin{cases}
vv'=1 &\\ 
 y'=yv^2&  \\ 
\end{cases}.$$
So, an open covering of $\F$ is given by the open embeddings
\begin{align*}
\rho_1 :U\times \mathbb{P}^1&\rightarrow \F\\
(v,[x:y])&\mapsto ([x:y:yv^2],[1:v])
\end{align*} 
and \begin{align*}
\rho_2 :U'\times \mathbb{P}^1&\rightarrow \F\\
(v',[x':y'])&\mapsto ([x':y'v'^2:y'],[v':1]),
\end{align*} which glue to give an isomorphism $\rho:\;\mathbb{P}(\mathcal{O}_{\mathbb{P}^1}(2)\oplus \mathcal{O}_{\mathbb{P}^1})\rightarrow \F$.
\end{proof}

Now, the algebraic group $H$ acts on the second Hirzebruch surface  $$\F = \{ ([x:y:z],[u:v]) \in \mathbb{P}^2 \times \mathbb{P}^1 | yv^2=zu^2\} $$ in the following manner: for $p=([x:y:z],[u:v]) \in \F$ and $g=\left((a_0,a_1,a_2)^t, \begin{pmatrix}
a &b \\c 
 &d 
\end{pmatrix}\right) \in H$, 
$$g.p=\begin{cases}
\left (\left [xu^2+y(a_0v^2+a_1uv+a_2u^2):y(au+bv)^2:y(cu+dv)^2  \right ] ,\left [au+bv:cu+dv  \right ]  \right )& \text{ if } u \not= 0 \\ 
 \left (\left [xv^2+z(a_0v^2+a_1uv+a_2u^2):z(au+bv)^2:z(cu+dv)^2  \right ] ,\left [au+bv:cu+dv  \right ]  \right )& \text{ if } v\not=0 
\end{cases}.$$
The following theorem is well-known (see [1, Section 6.1]).
\begin{thm} The group of automorphisms of $\F$ is exactly the quotient of $H$ by the subgroup $I$ consisting of diagonal matrices of the form $ \begin{pmatrix}
\mu & 0\\ 
0 & \mu
\end{pmatrix}$  where $\mu \in \mathbb{C}$ such that $\mu^2=1$.
\end{thm}

\section{A formal semi-universal deformation of $\F$ and formal vector fields on it}
\subsection{Construction of the semi-universal deformation of $\F$} 
We shall follow the construction given in [4, Example 1.2.2.(iii)]. Consider two copies of $\mathbb{C}\times\mathbb{C}\times \mathbb{P}^1$ given by $W:=\pr(\mathbb{C}[t,v,x,y])$ and $W':=\pr(\mathbb{C}[t',v',x',y'])$ (note that these two rings are graded with respect to $x,y$ and $x',y'$, respectively). Consider the affine subsets $\Sp(\mathbb{C}[t,v,y])\subset W$, $\Sp(\mathbb{C}[t',v',y']) \subset W'$ and then glue them along the open subsets

$$\Sp(\mathbb{C}[t,v,v^{-1},y]) \subset \Sp(\mathbb{C}[t,v,y])$$ and $$\Sp(\mathbb{C}[t',v',v'^{-1},y']) \subset \Sp(\mathbb{C}[t',v',y'])$$ by the rules
\begin{equation}\begin{cases}
vv'=1 & \\ 
x'=x & \\
 y' =yv^2-tvx&  \\ 
t'=t.
\end{cases}
\end{equation}
This gives a gluing of $W$ and $W'$ along $$\pr(\mathbb{C}[t,v,v^{-1},x,y])\text{ and }\pr(\mathbb{C}[t',v',v'^{-1},x',y']).$$ We denote the resulting scheme by $\mathcal{W}$. In other words, if we let $(t,v,[x:y])$ and $(t',v', [x':y'])$ be the coordinates on $W=\mathbb{C}\times\mathbb{C}\times \mathbb{P}^1$ and on $W'=\mathbb{C}\times\mathbb{C}\times \mathbb{P}^1$, respectively. Then $W$ is obtained by glue $W$ and $W'$ according to the rules $(3.1)$. Now, let $\pi:$ $\mathcal{W}\rightarrow \mathbb{C}$ be the morphism induced by the projections.
\begin{thm} The familly 
$\pi:$ $\mathcal{W}\rightarrow \mathbb{C}=\Sp(\mathbb{C}[t])$ is a semi-universal deformation of $\F$. Moreover,
$$\pi^{-1}(t)=\begin{cases}
 & \F\text{ if } t=0 \\ 
 & \mathbb{P}^1 \times \mathbb{P}^1 \text{ otherwise}.
\end{cases}$$
\end{thm}
\begin{proof}
The map $\pi$ is obviously surjective by construction. Since $\pi$ is locally a projection, it is a flat morphism. Moreover, by Proposition 2.1, $\mathcal{W}_0=\pi^{-1}(0)=\F.$ Then $\pi:$ $\mathcal{W}\rightarrow \mathbb{C}$ is a deformation of $\F$. Next, let $\mathcal{W}^*=\pi^{-1}(\mathbb{C}^*)$ and $\pi^*:\;\mathcal{W}^* \rightarrow \mathbb{C}^* $ is the restriction of $\pi$ on $\mathbb{C}^*=\Sp(\mathbb{C}[t,t^{-1}])$. We shall prove that $\mathcal{W}^*$ is in fact isomorphic to $\mathbb{C}^*\times \mathbb{P}^1\times \mathbb{P}^1$. Indeed, consider the following open embeddings 
\begin{align*}
\phi :\mathbb{C}^*\times \mathbb{C}\times \mathbb{P}^1&\rightarrow    \mathbb{C}^*\times \mathbb{P}^1\times \mathbb{P}^1\\
(t,v,[x:y])&\mapsto (t,[1:v],[ty:vy-tx])
\end{align*} 
and 

\begin{align*}
\phi':\mathbb{C}^*\times \mathbb{C}\times \mathbb{P}^1&\rightarrow\mathbb{C}^*\times \mathbb{P}^1\times \mathbb{P}^1\\
(t',v',[x':y'])&\mapsto (t',[v':1],[t'v'y'+t'^2x':y']).
\end{align*} 
By the gluing condition $(3.1)$, we have that
\begin{align*}
(t',[v':1],[t'v'y'+t'^2x':y']) &=(t',[v':1],[t'v'(yv^2-tvx)+t'^2x':yv^2-tvx]) \\ 
 &= (t',[v':1],[t'yv:yv^2-tvx]) \\ 
 &= (t,[1:v],[ty:yv-tx]).
\end{align*}
Hence, the above two morphisms glue to give an isomorphism \begin{center}
\begin{tikzpicture}[every node/.style={midway}]
  \matrix[column sep={8em,between origins}, row sep={3em}] at (0,0) {
    \node(Y){$\mathcal{W}^*$} ; & \node(X) {$\mathbb{C}^*\times \mathbb{P}^1 \times \mathbb{P}^1$}; \\
    \node(M) { }; & \node (N) {$\mathbb{C}^*$,};\\
  };
  \draw[->] (Y) -- (X) node[anchor=south]  {$\cong$};
  \draw[->] (X) -- (N) node[anchor=west] {$\proj_1$};
  \draw[->] (Y) -- (N) node[anchor=north] {$\pi^*$};
\end{tikzpicture}
\end{center}
which means precisely that $\pi^*:\;\mathcal{W}^* \rightarrow \mathbb{C}^* $ is the trivial family whose fibers are all isomorphic to $\mathbb{P}_1\times \mathbb{P}_1$. In particular, for $t\in \mathbb{C}^*$, $\pi^{-1}(t)=(\pi^*)^{-1}(t)=\mathbb{P}_1\times \mathbb{P}_1 $.

It remains to prove that the family $\pi:$ $\mathcal{W}\rightarrow \mathbb{C}$ is actually semi-universal.  One way to see it is to compute the Kodaira-Spencer map $\mathcal{K}_{\pi,0}$ of $\pi$ at $0$. This map is uniquely determined by the element $\mathcal{K}_{\pi,0}(\frac{d}{dt})$ in $H^1(\F,\mathcal{T}_{\F})$. By definition, $\mathcal{K}_{\pi,0}(\frac{d}{dt})$ represents the first order deformation of $\F$, obtained by gluing $W_0:=\pr(\mathbb{C}[\epsilon,v,x,y])$ and  $W_0':=\pr(\mathbb{C}[\epsilon,v',x',y'])$ along  $\pr(\mathbb{C}[\epsilon,v,v^{-1},x,y])$ and  $\pr(\mathbb{C}[\epsilon,v',v'^{-1},x',y'])$ by the rules 
$$\begin{cases}
vv'=1 & \\ 
 y' =yv^2-\epsilon v 
\end{cases},$$  where $\mathbb{C}[\epsilon]$ is the ring of complex dual numbers. Hence, $\mathcal{K}_{\pi,0}(\frac{d}{dt}) \in H^1(\mathcal{U},\mathcal{T}_{\F} )$ is the $1$-cocycle which corresponds to the vector field $\lbrace -v\frac{\partial}{\partial y} \rbrace$ on  $W_0 \cap W_0'$, where $\mathcal{U}$ is the covering $\lbrace W_0,W_0'\rbrace$. By [3, Example B.11(iii)], we see that $\lbrace -v\frac{\partial}{\partial y} \rbrace$ is nonzero and $\dim_{\mathbb{C}}H^1(\F,\mathcal{T}_{\F})=1$. Thus, the Kodaira-Spencer map is an isomorphism and so $\pi:$ $\mathcal{W}\rightarrow \mathbb{C}$ is semi-universal.
\end{proof}
Another useful presentation of $\mathcal{W}$ is given as follows.

\begin{prop} The scheme $\mathcal{W}$ is isomorphic to the surface
$$\mathcal{X}:=\{ ([x:y:z],[u:v],t)  \in \mathbb{P}^2 \times \mathbb{P}^1 \times \mathbb{C} \text{ }|\text{ } yv^2-zu^2 -txuv=0\} .
$$
\end{prop}
\begin{proof}
We have an open covering of $\mathcal{X}$ given by the open embeddings
\begin{align*}
\rho_1 :\mathbb{C}\times \mathbb{C}\times \mathbb{P}^1&\rightarrow \mathcal{X}\\
(t,v,[x:y])&\mapsto ([x:y:yv^2-tv],[1:v],t)
\end{align*} 
and \begin{align*}
\rho_2 :\mathbb{C}\times \mathbb{C}\times \mathbb{P}^1&\rightarrow \mathcal{X}\\
(t',v',[x':y'])&\mapsto ([x':y'v'^2+t'v':y'],[v':1],t)
\end{align*} which glue to give an isomorphism $\mathcal{W}\overset{\cong}{\rightarrow} \mathcal{X}$.
\end{proof}
\begin{rem} By Proposition 2.1 and by Proposition 3.1, from now on, we use interchangeably between $\F, \mathcal{X}$ and 
$\mathbb{P}(\mathcal{O}_{\mathbb{P}^1}(2)\oplus \mathcal{O}_{\mathbb{P}^1}),\mathcal{W}$, respectively.
\end{rem}

\subsection{Formal vector fields on the formal semi-universal deformation of $\F$}
The formal deformation associated to $\mathcal{W}$, $\widehat{\pi}:\widehat{\mathcal{W}} \rightarrow \spf(\mathbb{C}[[t]])$ is a formal semi-universal deformation of $\F$ (here $\mathbb{C}[[t]]$ is the ring of formal power series in the variable $t$). We will give explicit descriptions of formal vector fields on  $\widehat{\mathcal{W}}$. Consider the covering $\{W,W'\}$ where $W:=\pr(\mathbb{C}[t,v,x,y])$ and $W':=\pr(\mathbb{C}[t',v',x',y'])$, as before. A formal vector field on $W$ is of the form
\begin{equation}
 g_1(v,t)\frac{\partial }{\partial v}+(\alpha_1(v,t)y^2+\beta_1(v,t)y+\gamma_1(v,t))\frac{\partial }{\partial y}+k_1(t)\frac{\partial }{\partial t}\end{equation} where $g_1,\alpha_1,\beta_1,\gamma_1,k_1$ are formal power series in the variable $t$. Likewise, a formal vector field on $W'$ is of the form
 
\begin{equation}
 g_2(v',t')\frac{\partial }{\partial v'}+(\alpha_2(v',t')y'^2+\beta_2(v',t')y'+\gamma_2(v',t'))\frac{\partial }{\partial y'}+k_2(t')\frac{\partial }{\partial t'}
\end{equation} where $g_2,\alpha_2,\beta_2,\gamma_2,k_2$ are formal power series in the variable $t'$. Therefore, a vector field on $\mathcal{W}$ which is of the form $(3.2)$ on $W $ and of the form $(3.3)$ on $W' $ must satisfy the relation 

\begin{equation}
\begin{matrix}
g_1(v,t)\frac{\partial }{\partial v}+(\alpha_1(v,t)y^2+\beta_1(v,t)y+\gamma_1(v,t))\frac{\partial }{\partial y}+k_1(t)\frac{\partial }{\partial t} \text{ } \text{ } \text{ } \text{ } \text{ } \text{ } \text{ } \text{ } \text{ } \text{ } \text{ } \text{ } \text{ } \text{ } \text{ } \text{ } \text{ } \text{ } \text{ } \text{ } \text{ } \text{ } \text{ } \\ 
\text{ } \text{ } \text{ } \text{ } \text{ }\text{ } \text{ }=g_2(v',t')\frac{\partial }{\partial v'}+(\alpha_2(v',t')y'^2+\beta_2(v',t')y'+\gamma_2(v',t'))\frac{\partial }{\partial y'}+k_2(t')\frac{\partial }{\partial t'}
\end{matrix}
\end{equation}
on the overlapping open set  $W \cap W'$.

\begin{lem}
A global formal vector field on  $\widehat{\mathcal{W}}$ whose restriction on $W$  is $$g_1(v,t)\frac{\partial }{\partial v}+(\alpha_1(v,t)y^2+\beta_1(v,t)y+\gamma_1(v,t))\frac{\partial }{\partial y}+k_1(t)\frac{\partial }{\partial t}$$ must satisfy the following \begin{equation}
\begin{cases}
 g_1(v,t)=A(t)v^2+B(t)v+C(t)\\ 
 \alpha_1(v,t)=a(t)v^2+b(t)v+c(t)\\ 
 \beta_1(v,t)=-2[a(t)t+A(t)]v+e(t)\\ 
\gamma_1(v,t)=t^2a(t)+tA(t)
\end{cases}
\end{equation} where $A,B,C,a,b,c,e, k_1$ are formal power series in the variable $t$ with a relation 
\begin{equation}
b(t)t^2+e(t)t+B(t)t-k_1(t)=0. 
\end{equation}
\end{lem}

\begin{proof}
 By $(3.1)$, we have
\begin{equation}\begin{cases}
y=v'^2y'+tv'\\ 
v=\frac{1}{v'}\\ t=t'\\
\partial_v=-v'^2\partial_{v'}+(2y'v'+t)\partial_{y'}\\ \partial_{y}=\frac{1}{v'^2}\partial_{y'}\\\partial_{t'}=-\frac{1}{v'}\partial_{y'}+\partial_{t'}.
\end{cases}\end{equation}
Substituting  $(3.7)$ into the left hand side of $(3.4)$ and equalizing, we get that
\begin{equation}
\begin{cases}
 g_2(v',t')=-v'^2g_1(\frac{1}{v'},t')\\ 
 \alpha_2(v',t')=v'^2\alpha_1(\frac{1}{v'},t')\\ 
 \beta_2(v',t')=2t'v'\alpha_1(\frac{1}{v'},t')+\beta_1(\frac{1}{v'},t')+2v'g_1(\frac{1}{v'},t')\\ 
\gamma_2(v',t')=t'^2\alpha_1(\frac{1}{v'},t')+\frac{t'}{v'}\beta_1(\frac{1}{v'},t')+\frac{1}{v'^2}\gamma_1(v',t')+t'g_1(\frac{1}{v'},t')-\frac{k_1(t')}{v'},
\end{cases}
\end{equation}
which implies that
$$
\begin{cases}
 g_1(v,t)=A(t)v^2+B(t)v+C(t)\\ 
 \alpha_1(v,t)=a(t)v^2+b(t)v+c(t)\\ 
 \beta_1(v,t)=-2[a(t)t+A(t)]v+e(t)\\ 
\gamma_1(v,t)=t^2a(t)+tA(t),
\end{cases}
$$ where $A,B,C,a,b,c,e$ are formal power series in the variable $t$ with a relation 
$$
b(t)t^2+e(t)t+B(t)t-k_1(t)=0. $$
This constraint comes from the coefficient of $\frac{1}{v'}$ in the fourth equation in $(3.8)$.
\end{proof}
\begin{rem} If $t=0$ then $(3.5)$ becomes 
$$\begin{cases}
 g_1(v)=Av^2+Bv+C\\ 
 \alpha_1(v)=av^2+bv+c\\ 
 \beta_1(v)=-2Av+e\\ 
\gamma_1(v,t)=0
\end{cases}$$ which agrees with Kodaira's calculations of vector fields on $\mathcal{W}_0=\F$ (see [3, Page 75]). In particular, we have seven linearly independent vector fields on $\F$. If $t$ is non-zero and fixed then we have six linearly independent vector fields on the fiber $\mathcal{W}_t$, which is due to the existence of the relation $(3.6)$.

\end{rem}

\section{The non-existence of $G$-equivariant structure on the formal semi-universal deformation} 
The Lie algebra of $G:= \text{Aut}(\F)$ is $\mathbb{C}^3\times M(2,\mathbb{C})$, which is evidently $7$-dimensional . A $\mathbb{C}$-basis of $\text{Lie(G)}$ is given by the following elements
$$\begin{cases}
e_1=(1,0,0)\times \begin{pmatrix}
0 &0 \\ 
0 &0 
\end{pmatrix},\text{ }e_2=(0,0,1)\times \begin{pmatrix}
0 &0 \\ 
0 &0 
\end{pmatrix},\text{ }e_3=(0,0,0)\times \begin{pmatrix}
0 &0 \\ 
1 &0 
\end{pmatrix}, &  \\ 
e_4=(0,1,0)\times \begin{pmatrix}
0 &0 \\ 
0 &0 
\end{pmatrix},\text{ }e_5=(0,0,0)\times \begin{pmatrix}
1 &0 \\ 
0 &0 
\end{pmatrix},\text{ }e_6=(0,0,0)\times \begin{pmatrix}
0 &0 \\ 
0 &1 
\end{pmatrix},&  \\ 
e_7=(0,0,0)\times \begin{pmatrix}
0 &1 \\ 
0 &0
\end{pmatrix}.
\end{cases}$$
Then the $G$-action gives us $7$ vector fields $E_1',\ldots,E_7'$ on $\F$ with the relations 

\begin{align*}
 &\begin{cases}
[E_1',E_2']=0 \\ 
 [E_1',E_3']=-2E_4' \\ 
[E_1',E_4']=0  \\ 
[E_1',E_5']=0  \\ 
[E_1',E_6']=-2E_1'\\
[E_1',E_7']=0,  
\end{cases}
\begin{cases}
[E_2',E_3']=0  \\ 
[E_2',E_4']=0  \\ 
[E_2',E_5']=-2E_2'  \\ 
[E_2',E_6']=0 \\
[E_2',E_7']=-2E_4', 
\end{cases}\begin{cases}
[E_3',E_4']=E_2'  \\ 
[E_3',E_5']=-E_3'  \\ 
[E_3',E_6']=E_3'\\ 
[E_3',E_7']=E_5'-E_6' ,
\end{cases}\\ 
 & \begin{cases}
[E_4',E_5']=-E_4' \\ 
[E_4',E_6']=-E_4' \\ 
[E_4',E_7']=-E_1',
\end{cases}\text{ }\begin{cases}
[E_5',E_6']=0 \\ 
[E_5',E_7']=-E_7', 
\end{cases}\text{ }\text{ }[E_6',E_7']=E_7'.
\end{align*}

Now, we are in the position to prove the main result of this paper. Suppose that the $G$-action extends on  $\widehat{\mathcal{W}}$ . This implies that we also have $7$ formal vector fields $E_1,E_2,E_3,E_4,E_5,E_6,E_7$ on  $\widehat{\mathcal{W}}$ with the following Lie bracket constraints

\begin{align*}
 &\begin{cases}
[E_1,E_2]=0 \\ 
 [E_1,E_3]=-2E_4 \\ 
[E_1,E_4]=0  \\ 
[E_1,E_5]=0  \\ 
[E_1,E_6]=-2E_1\\
[E_1,E_7]=0,  
\end{cases}
\begin{cases}
[E_2,E_3]=0  \\ 
[E_2,E_4]=0  \\ 
[E_2,E_5]=-2E_2  \\ 
[E_2,E_6]=0 \\
[E_2,E_7]=-2E_4, 
\end{cases}\begin{cases}
[E_3,E_4]=E_2  \\ 
[E_3,E_5]=-E_3  \\ 
[E_3,E_6]=E_3\\ 
[E_3,E_7]=E_5-E_6,
\end{cases}\\ 
 & \begin{cases}
[E_4,E_5]=-E_4 \\ 
[E_4,E_6]=-E_4 \\ 
[E_4,E_7]=-E_1,
\end{cases}\text{ }\begin{cases}
[E_5,E_6]=0 \\ 
[E_5,E_7]=-E_7, 
\end{cases}\text{ }\text{ }[E_6,E_7]=E_7.
\end{align*}
These vector fields form a Lie subalgebra, denoted by $\mathfrak{g}$, of the Lie algebra of formal vector fields on  $\widehat{\mathcal{W}}$. Of course, the restriction of $E_i$ on the central fiber is nothing but $E_i'$ ($i=1,\ldots,7$).

From the previous section, we can assume that our seven vector fields are of the form 
$$E_i=g_i(v,t)\frac{\partial }{\partial v}+(\alpha_i(v,t)y^2+\beta_i(v,t)y+\gamma_i(v,t))\frac{\partial }{\partial y}+k_i(t)\frac{\partial }{\partial t},$$(cf. Lemma 3.1) where $A,B,C,a,b,c,e$ are formal power series in $t$ ($i=1,\ldots,7$).

\begin{thm}
The action of $G$ on $\F$ does not extend to the formal semi-universal deformation  $\widehat{\mathcal{W}}$, where $G$ is the automorphism group of $\F$.
\end{thm}
\begin{proof}

We denote by $\mathfrak{v}$ the Lie algebra of formal vector fields in one variable $t$. Let $\delta:$ $\mathfrak{g}\rightarrow \mathfrak{v}$ be the map which sends $$g_i(v,t)\frac{\partial }{\partial v}+(\alpha_i(v,t)y^2+\beta_i(v,t)y+\gamma_i(v,t))\frac{\partial }{\partial y}+k_i(t)\frac{\partial }{\partial t}$$ to $$k_i(t)\frac{\partial }{\partial t},$$ for $i=1,\ldots, 7$. Since, the first two components $\frac{\partial }{\partial v}$ and $\frac{\partial }{\partial y}$ contribute nothing to the component $\frac{\partial }{\partial t}$ in the Lie bracket then $\delta$ is a well-defined Lie homomorphism. Set $F_i:=\delta(E_i)=k_i(t)\frac{\partial }{\partial t}$ ($i=1,\ldots,7$). Note that the seven formal vector fields $F_i$ ($i=1,\ldots, 7$) are nothing but those induced by the $G$-action on the base $\spf(\mathbb{C}[[t]])$ (cf. the last paragraph of Section 1). Observe also that $\mathfrak{v}$ can be equipped with a filtration $F$ given by the vanishing order at $0$ and we have two well-known facts
$$[F^p\mathfrak{v},F^p\mathfrak{v}]\subset F^{2p}\mathfrak{v}, \text{ and } [F^p\mathfrak{v},F^q\mathfrak{v}]\subset F^{p+q-1}\mathfrak{v},$$ for $p,q \geq 1$. Furthermore, the vanishing order of all $k_i$ at $0$ is at least $1$. Let $k_i(t)=\sum_{j=1}^{\infty}a^i_jt^j$ ($i=1,2,4,5$). Using the first fact and the Lie relations induced by $\delta$: 
$$\begin{cases}
[F_1,F_6]=-2F_1 \\ 
[F_2,F_5]=-2F_2\\ 
[F_1,F_3]=-2F_4,
\end{cases}$$
 we obtain $a_1^1=a_1^2=a_1^4=0$. Suppose that $k_4(t)$ is not identically zero, then there exists $j^*\geq 2$ such that $a_{j^*}^4$ is nonzero. By computing explicitly the Lie relation $[F_4,F_5]=-F_4$ in terms of power series in $t$ and then by equalizing coefficients, we get that 
$$a_j^4[(j-1)a_1^5-1]=0,$$ for all $j\geq 2$. Thus, $a_1^5=\frac{1}{j^*-1}$, which is clearly nonzero. A similar computation for the relation $[F_1,F_5]=0$ gives
$$(j-1)a_j^1a_1^5=0,$$ for all $j\geq 2$. Hence, all $a_j^1=0$ so that $k_1(t)=0$. By the relation $[F_1,F_3]=-2F_4$, we deduce that $k_4(t)=0$, a contradiction. Therefore, $k_4(t)=0$. From the relations $[F_3,F_4]=F_2$ and $[F_4,F_7]=-F_1$, we obtain that $k_2(t)=0$ and $k_1(t)=0$. As a sequence, $E_1,E_2$, and $E_4$ do not have the component $\frac{\partial }{\partial t}$.

In addition, by the proof of Theorem 3.1, as a scheme over $\mathbb{C}$, $$\mathcal{W}^*\cong \mathbb{C}^*\times  \mathbb{P}^1\times \mathbb{P}^1.$$  Then, $$\mathcal{W}^*\cong \mathbb{P}^1_{L}\times \mathbb{P}^1_{L},$$ as a scheme over $L$, where $L:=\mathbb{C}[t,t^{-1}]$ and $\mathbb{P}^1_{L}$ is the $1$-projective space over $L$. Therefore, the generic fiber $\widehat{\mathcal{W}}^*$ of $\widehat{\mathcal{W}}$ is isomorphic to  $\mathcal{W}^* \times_{\Sp(\mathbb{C}[t,t^{-1}])} \Sp(\mathbb{C}[[t,t^{-1}]])=\mathbb{P}^1_{K}\times \mathbb{P}^1_{K}$, as a scheme over $K$, where $K$ is the field of Laurent formal power series $\mathbb{C}[[t,t^{-1}]]$. Now, by restricting on the generic fiber of $\widehat{\mathcal{W}}$, we obtain that $E_1,E_2$, and $E_4$ are formal vector fields on  $\widehat{\mathcal{W}}^*$, considered as a $\mathbb{C}$-scheme. However, by the first paragraph, we have proved that there is no component $\frac{\partial }{\partial t}$ in the expression of $E_i$ ($i=1,2,4$). So, if we think of $E_1$, $E_2$ and $E_4$ as vector fields with coefficients in $K$, then they are definitely vector fields on $\widehat{\mathcal{W}}^*$, regarded as a scheme over $K$. Note that the Lie algebra of vector fields on $\mathbb{P}^1_K\times\mathbb{P}^1_K$ is isomorphic to $\mathfrak{sl}_2(K)\times \mathfrak{sl}_2(K)$, where $\mathfrak{sl}_2(K)$ is the special linear group. This means that there exists a $3$-dimensional abelian Lie subalgebra of $\mathfrak{sl}_2(K)\times \mathfrak{sl}_2(K)$. The image of that subalgebra under one of the two canonical projections of the product $\mathfrak{sl}_2(K)\times \mathfrak{sl}_2(K)$  provides a $2$-dimensional abelian Lie subalgebra in $\mathfrak{sl}_2(K)$. This is a contradiction since $\text{rank}(\mathfrak{sl}_2(K))$ is only $1$.
\end{proof}
\begin{rem} A naturally posed question is if the $G$-action extends to $\mathcal{W}_n$ over $\Sp(\mathbb{C}[t]/(t^{n+1}))$ for small value $n$. Although the above proof does not give any clue to reply to this question, the answer is yes for $n=1$. More general, if $G$ is an algebraic group acting algebraically on a projective variety $X$ and $\pi: \mathcal{X} \rightarrow S$ is the semi-universal deformation of $X$ then the $G$-action on $X$ certainly extends up to the first infinitesimal deformation $\mathcal{X}_1$ over $S_1$. This follows easily from the semi-universality of the family $\pi: \mathcal{X} \rightarrow S$.
Unfortunately, our example turns out to be the worst case. More precisely, we can even show that the $G$-action on $\mathbb{F}_2$ can not extend to the second infinitesimal $\mathcal{W}_2$ over $\Sp(\mathbb{C}[t]/(t^3))$ by extending $E'_i$ ($i=1,\dots,7$) together with their Lie bracket relations, order by order with respect to $t$. However, the computations are somewhat lengthy and complicated.
\end{rem}

\bibliographystyle{amsplain}

\end{document}